\documentclass[11pt]{amsart}

\author{Carlo Sanna}
\address{Universit\`a degli Studi di Torino\\Department of Mathematics\\Turin, Italy}
\email{carlo.sanna.dev@gmail.com}

\title{Covering an arithmetic progression with \mbox{geometric progressions} and vice versa}

\usepackage{amsmath}
\usepackage{amssymb}
\usepackage{amsthm}
\usepackage{geometry}
\usepackage{hyperref}
\hypersetup{colorlinks=true}

\newtheorem{thm}{Theorem}[section]
\newtheorem{lem}[thm]{Lemma}

\DeclareMathAlphabet{\mathpzc}{OT1}{pzc}{m}{it}

\def\N{\mathbf{N}}
\def\Q{\mathbf{Q}}
\def\A{\mathpzc{A}}
\def\G{\mathpzc{G}}
\def\B{\mathpzc{B}}
\def\S{\mathpzc{S}}
                                 
\begin{document}

\begin{abstract}
We show that there exists a positive constant $C$ such that the following holds: Given an infinite arithmetic progression $\A$ of real numbers and a sufficiently large integer $n$ (depending on $\A$), there needs at least $Cn$ geometric progressions to cover the first $n$ terms of $\A$.
A similar result is presented, with the role of arithmetic and geometric progressions reversed.
\end{abstract}

\maketitle

\section{Introduction}

Arithmetic and geometric progressions are always an active research topic in \mbox{Number Theory}.
In particular, problems concerning arithmetic progressions and covering, mostly over the integers, are well studied (for example, see \cite{Sun95}).
For $v \geq 0$ and $d > 0$, let
\begin{equation*}
\A(v,d) := \{v, v + d, v + 2d, v + 3d, \ldots\}
\end{equation*}
be the arithmetic progression with first term $v$ and common difference $d$.
Also, for $u > 0$ and $q > 1$, let
\begin{equation*}
\G(u,q) := \{u, uq, uq^2, uq^3, \ldots\}
\end{equation*}
be the geometric progression with first term $u$ and ratio $q$.
Furthermore, for a positive integer $n$, let $\A^{(n)}$, respectively $\G^{(n)}$, be the set of the first $n$ terms of the arithmetic progression $\A$, respectively the geometric progression $\G$.
Now, for a finite set $\S$ of nonnegative real numbers, denote by $g(\S)$ the least positive integer $h$ such that there exist $h$ geometric progressions $\G_1, \ldots, \G_h$ covering $\S$, i.e., $\S \subseteq \bigcup_{i=1}^h \G_i$.
Similarly, denote by $a(\S)$ the least positive integer $h$ such that there exist $h$ arithmetic progressions covering $\S$.
Since given any two distinct nonnegative real numbers there is an arithmetic progression, respectively a geometric progression, containing them; it follows easily that $a(\S), g(\S) \leq (|\S| + 1) / 2$.
On the other hand, obviously, $a(\A^{(n)}) = g(\G^{(n)}) = 1$ for each arithmetic progression $\A$ and each geometric progression $\G$.
We are interested in lower bounds for $g(\A^{(n)})$ and $a(\G^{(n)})$.
Our first result is the following theorem.

\begin{thm}\label{thm:geometric_cover}
There exists a positive constant $C_1$ such that for each arithmetic progression $\A = \A(v,d)$ it results $g(\A^{(n)}) \geq C_1 n$ for $n$ sufficiently large (how large depending only on $v / d$).
In particular, we can take $C_1 = 1 / \pi^2$.
\end{thm}

Regarding a lower bound for $a(\G^{(n)})$, with $\G = \G(u,q)$, the situation is a little bit different.
In fact, we need to distinguish according to whether $q$ is a root of a rational number $>1$ or not.

\begin{thm}\label{thm:arithmetic_rational_cover}
Let $q = r^{1/m}$ with $r > 1$ a rational number and $m$ a positive integer such that $q^{m^\prime}$ is irrational for any positive integer $m^\prime < m$.
Then $a(\G^{(n)}) \leq m$ for each geometric progression $\G = \G(u,q)$ and each integer $n \geq 1$, with equality if $n \geq 2m$.
\end{thm}

\begin{thm}\label{thm:arithmetic_irrational_cover}
There exists a positive constant $C_2$ such that if $q \neq r^{1/m}$ for all rationals $r > 1$ and all positive integers $m$, then $a(\G^{(n)}) \geq C_2 n$ for each geometric progression $\G$ and each integer $n \geq 1$.
In particular, we can take $C_2 = 1 / 6$.
\end{thm}

A natural question, open to us, is the evaluation of the best constants $C_1$ and $C_2$ in Theorem \ref{thm:geometric_cover} and Theorem \ref{thm:arithmetic_irrational_cover},
i.e., to find
\begin{equation*}
\inf_{\A} \liminf_{n \to \infty} \frac{g(\A^{(n)})}{n} \quad \mbox{and} \quad \inf_{\G} \liminf_{n \to \infty} \frac{a(\G^{(n)})}{n} ,
\end{equation*}
where $\A$ runs over all the arithmetic progressions and $\G$ runs over all the geometric progressions $\G = \G(u,q)$, with ratio $q$ not a root of a rational number $>1$.~
The results above give $1/\pi^2 \leq C_1 \leq 1/2$ and $1/6 \leq C_2 \leq 1/2$.

\subsection*{Notation}
Hereafter, $\N$ denotes the set of positive integers and $\N_0 := \N \cup \{0\}$.
The letter $p$ is reserved for prime numbers and $\upsilon_p(\cdot)$ denotes the $p\,$-adic valuation over the field of rational numbers $\Q$.

\section{Preliminaries}

The fundamental tool for our results is the following theorem of A. Dubickas and J. Jankauskas, regarding the intersection of arithmetic and geometric progressions \cite[Theorem 3 and 4]{DJ10}.

\begin{thm}\label{thm:intersection}
Suppose that the ratio $q > 1$ is not or the form $r^{1/m}$, with $r > 1$ a rational number and $m \in \N$, then $|\A \cap \G| \leq 6$ for each arithmetic progression $\A$ and each $\G=\G(u,q)$.
\end{thm}

If the ratio $q$ of the geometric progression $\G$ is a root of a rational number $>1$, then, without further assumptions,
$|\A \cap \G|$ can take any nonnegative integer value, or even be infinite \cite[Theorem 1 and 2]{DJ10}.
However, we have the following:

\begin{lem}\label{lem:rational}
Suppose that $q = r^{1/m}$, with $r > 1$ rational and $m \in \N$ such that $q^{m^\prime}$ is irrational for any positive integer $m^\prime < m$.
If $|\A(v,d) \cap \G(u,q)| \geq 3$ then $v / d$ is rational and $u / d = s q^{-\ell}$ for some $s \in \Q$ and some $\ell \in \{0,1,\ldots,m-1\}$.
Moreover, for each $uq^k \in \A(v,d) \cap \G(u,q)$ it results $k \equiv \ell$ (mod $m$).
\end{lem}
\begin{proof}
Since $|\A(v,d) \cap \G(u,q)| \geq 3$, there exist $k_1, k_2, k_3 \in \N_0$ pairwise distinct and such that $uq^{k_i} = v + dh_i$, with $h_i \in \N_0$ ($i = 1,2,3$).
Set $t := v / d$ and $\xi := u / d$, so that $\xi q^{k_i} = t + h_i$ for each $i$.
Then, $q^{k_1} \neq q^{k_3}$ and
\begin{equation*}
\frac{q^{k_1} - q^{k_2}}{q^{k_1} - q^{k_3}} = \frac{h_1 - h_2}{h_1 - h_3} \in \Q ,
\end{equation*}
so that $q^{k_1}, q^{k_2}, q^{k_3}$ are linearly dipendent over $\Q$.
Being $x^m - r$ the minimal polynomial of $q$ over the rationals (as it follows at once from our assumptions), we have that $q^0, q^1, \ldots, q^{m-1}$ are linearly independent over $\Q$.
It follows that at least two of $k_1, k_2, k_3$ lie in the same class modulo $m$.
Without loss of generality, we can assume $k_1 \equiv k_2$ (mod $m$), so that $q^{k_1 - k_2}$ is rational.
Now
\begin{equation*}
t + h_1 = \xi q^{k_1} = q^{k_1 - k_2} \xi q^{k_2} = q^{k_1 - k_2} (t + h_2) ,
\end{equation*}
thus, on the one hand,
\begin{equation*}
t = \frac{h_1 - q^{k_1 - k_2} h_2}{q^{k_1 - k_2} - 1} \in \Q ,
\end{equation*}
and on the other hand $\xi = (t + h_1)q^{-k_1} = s q^{-\ell}$ for some $s \in \Q$ and some $\ell \in \{0,1,\ldots,m-1\}$ such that $\ell \equiv k_1$ (mod $m$).
In conclusion, for each $uq^k \in \A(v,d) \cap \G(u,q)$ we have $\xi q^k = t + h$ for some $h \in \N_0$, so $q^{k-\ell} = (t + h) / s \in \Q$ and necessarily $k \equiv \ell$ (mod $m$).
\end{proof}

Finally, we need to state the following lemma about the asymptotic density of squarefree integers in an arithmetic progression \cite{Pra58}.

\begin{lem}\label{lem:squarefree}
Let $a, b$ be integers with $b \geq 1$ and $\gcd(a,b)=1$. Then
\begin{equation*}
|\{k \in \N_0 : k \leq x \emph{ and } a + bk \emph{ is squarefree}\}| \sim \frac{6}{\pi^2} \prod_{p \mid b} \left(1-\frac1{p^2}\right)^{-1} x
\end{equation*}
as $x \to \infty$.
\end{lem}

\section{Proof of Theorem \ref{thm:geometric_cover}}

Let $\A = \A(v,d)$ and $n \in \N$. 
For the sake of brevity, set $g := g(\A^{(n)})$ and let $\G_1, \ldots, \G_g$ be geometric progressions such that $\A^{(n)} \subseteq \bigcup_{i=1}^g \G_i$.
Suppose that $|\A \cap \G_i| \leq 6$ for $i=1,2,\ldots,g$. 
Then
\begin{equation*}
n = |\A^{(n)}| \leq \sum_{i=1}^g |\A^{(n)} \cap \G_i| \leq \sum_{i=1}^g |\A \cap \G_i| \leq 6g ,
\end{equation*}
with the result that $g \geq n/6 > n / \pi^2$.

Suppose now that there exists $i_0 \in \{1,2,\ldots,g\}$ such that $|\A \cap \G_{i_0}| > 6$.
For a moment, let $\G_{i_0} = \G(u,q)$.
It follows from Theorem \ref{thm:intersection} that $q = r^{1/m}$ for some rational number $r > 1$ and $m \in \N$.
In particular, we can assume that $q^{m^\prime}$ is irrational for any positive integer $m^\prime < m$.
Therefore, Lemma \ref{lem:rational} implies that $t := v / d$ is rational, $u / d = s q^{-\ell}$ with $s \in \Q$, $\ell \in \{0,1,\ldots,m-1\}$, and that for each $uq^k \in \A \cap \G_{i_0}$ we have $k \equiv \ell$ (mod $m$).
Since $t \geq 0$ is rational, we can write $t = a / b$, where $a \geq 0$ and $b \geq 1$ are relatively prime integers.
On the other hand, if $uq^k \in \A \cap \G_{i_0}$ then $k = mj + \ell$ and $uq^k = v + dh$ for some $j, h \in \N_0$.
As a consequence, $sr^j = t + h$ and $bsr^j = a + bh$.
Now we claim that there exist at most two $j \in \N_0$ such that $bsr^j$ is a squarefree integer.
In fact, since $r > 1$, there exists a prime $p$ such that $\upsilon_p(r) \neq 0$. 
So $\upsilon_p(bsr^j) = \upsilon_p(bs) + j \upsilon_p(r)$ is a strictly monotone function of $j$ and can take the values $0$ or $1$, which is a necessary condition for $bsr^j$ to be a squarefree integer, for at most two $j \in \N_0$.
Consequently, if we define
\begin{equation*}
\B^{(n)} := \{v + dh \in \A^{(n)} : a + bh \mbox{ is squarefree}\} ,
\end{equation*}
then $|\B^{(n)} \cap \G_{i_0}| \leq 2$. 
Note that the definition of $\B^{(n)}$ depends only on $v$, $d$ and $n$, so we can conclude that $|\B^{(n)} \cap \G_i| \leq 6$ for all $i=1,2,\ldots,g$.
In fact, on the one hand, if $|\A^{(n)} \cap \G_i| \leq 6$ then it is straightforward that $|\B^{(n)} \cap \G_i| \leq 6$, since $\B^{(n)} \subseteq \A^{(n)}$.
On the other hand, if $|\A^{(n)} \cap \G_i| > 6$ then we have proved that $|\B^{(n)} \cap \G_i| \leq 2$.
Now, Lemma \ref{lem:squarefree} yields
\begin{equation*}
|\B^{(n)}| \sim \frac{6}{\pi^2} \prod_{p \mid b} \left(1 - \frac1{p^2}\right)^{-1} n,
\end{equation*}
as $n \to \infty$, so that $|\B^{(n)}| \geq (6/\pi^2) n$ for $n$ sufficiently large, depending only on $a,b$, i.e., $t$.

In conclusion,
\begin{equation*}
\frac{6}{\pi^2} n \leq |\B^{(n)}| \leq \sum_{i=1}^g |\B^{(n)} \cap \G_i| \leq 6g ,
\end{equation*}
hence $g \geq n / \pi^2$, for sufficiently large $n$. This completes the proof.

\section{Proofs of Theorem \ref{thm:arithmetic_rational_cover} and \ref{thm:arithmetic_irrational_cover}}

Let $\G = \G(u,q)$ and $n \in \N$. For the sake of brevity, set $a := a(\G^{(n)})$ and let $\A_1, \ldots, \A_a$ be arithmetic progressions such that $\G^{(n)} \subseteq \bigcup_{i=1}^a \A_i$.
Suppose $q = r^{1/m}$, for a rational number $r > 1$ and $m \in \N$ such that $q^{m^\prime}$ is irrational for all positive integers $m^\prime < m$.
Since $r > 1$ is rational, we can write $r = r_1 / r_2$, where $r_1$ and $r_2$ are coprime positive integers.
Then, for $k=0,1,\ldots,n-1$, we have
\begin{equation*}
uq^k = uq^{(k \bmod m)} r^{\lfloor k / m \rfloor }  = 0 + \frac{uq^{(k \bmod m)}}{r_2^n} \cdot r_1^{\lfloor k / m \rfloor}r_2^{n  - \lfloor k / m \rfloor} \in \A(0, uq^{(k \bmod m)} / r_2^n) ,
\end{equation*}
so that $\G^{(n)} \subseteq \bigcup_{i=0}^{m-1} \A(0, uq^i / r_2^n)$ and $a \leq m$.
Suppose now $n \geq 2m$. 
We define the sets $J := \{1,2,\ldots,a\}$,
\begin{equation*}
J_1 := \{i \in J : \exists \, uq^{k_1}, uq^{k_2} \in \A_i \cap \G^{(n)} \mbox{ such that } k_1 \neq k_2, \; k_1 \equiv k_2 \!\!\! \pmod m\} ,
\end{equation*}
and $J_2 := J \setminus J_1$. 
Clearly, $\{J_1, J_2\}$ is a partition of $J$.
For $i \in J$, suppose that there exist $uq^{k_1}, uq^{k_2} \in \A_i \cap \G^{(n)}$ such that $k_1 < k_2$.
This implies that if $\A_i = \A(v,d)$ then 
\begin{equation*}
d = \tfrac1{s}(uq^{k_2} - uq^{k_1}),
\end{equation*}
for some positive integer $s$.
Furthermore, if $uq^k \in \A_i \cap \G^{(n)}$ then
\begin{equation*}
uq^k = uq^{k_1} + dh = uq^{k_1} + \tfrac{h}{s}(uq^{k_2} - uq^{k_1}) ,
\end{equation*}
for some integer $h$, hence
\begin{equation}\label{equ:q_k}
q^k = (1 - \tfrac{h}{s}) q^{k_1} + \tfrac{h}{s}q^{k_2} .
\end{equation}
On the one hand, if $i \in J_1$ then we can assume $k_1 \equiv k_2$ (mod $m$), thus $q^{k_1}$ and $q^{k_2}$ are linearly dependent over $\Q$.
From (\ref{equ:q_k}) it follows that also $q^k$ and $q^{k_1}$ are linearly dependent over $\Q$, i.e., $k \equiv k_1$ (mod $m$).
On the other hand, if $i \in J_2$ then we can assume $k_1 \not\equiv k_2$ (mod $m$), thus $q^{k_1}$ and $q^{k_2}$ are linearly independent over $\Q$.
From (\ref{equ:q_k}) it follows that necessarily $h = 0$ and $k=k_1$, or $h = s$ and $k = k_2$.
To summarize, we have found that for a fixed $i \in J_1$ it results that all the $uq^k \in \A_i \cap \G^{(n)}$ have $k$ in the same class modulo $m$, while for $i \in J_2$ we have $|\A_i \cap \G^{(n)}| \leq 2$.
As a consequence, if 
\begin{equation*}
R := \Big\{k_1 \in \{0,1,\ldots,m-1\} : uq^k \in \bigcup_{i \in J_1} (\A_i \cap \G^{(n)}) \mbox{ for some } k \equiv k_1 \!\!\! \pmod m\Big\} ,
\end{equation*}
then $|J_1| \geq |R|$. 
Also, if $k_1 \in \{0,1,\ldots,m-1\} \setminus R$ and $uq^k \in \G^{(n)}$, with $k \equiv k_1$ (mod $m$), then $uq^k \notin \bigcup_{i \in J_1} (\A_i \cap \G^{(n)})$.
But $uq^k \in \bigcup_{i \in J} (\A_i \cap \G^{(n)})$ and so $uq^k \in \bigcup_{i \in J_2} (\A_i \cap \G^{(n)})$.
Thus, it follows that
\begin{equation}\label{equ:inclusion}
\bigcup_{k_1 \in \{0,1,\ldots,m-1\} \setminus R} \{uq^k \in \G^{(n)} : k \equiv k_1 \!\!\! \pmod m\} \subseteq \bigcup_{i \in J_2} (\A_i \cap \G^{(n)}) .
\end{equation}
The set on the left hand side of (\ref{equ:inclusion}) is an union of $(m - |R|)$ pairwise disjoint sets, each of them has at least $\lfloor n / m \rfloor$ elements, so it has at least
\begin{equation*}
\lfloor n / m \rfloor (m - |R|) \geq 2(m - |R|)
\end{equation*}
elements. This and $(\ref{equ:inclusion})$ yield
\begin{equation*}
2 (m - |R|) \leq \left|\bigcup_{i \in J_2} (\A_i \cap \G^{(n)})\right| \leq \sum_{i \in J_2} |\A_i \cap \G^{(n)}| \leq 2|J_2| ,
\end{equation*}
so that $|J_2| \geq m - |R|$.
In conclusion,
\begin{equation*}
a = |J| = |J_1| + |J_2| \geq |R| + (m - |R|) = m ,
\end{equation*}
hence $a = m$. This completes the proof of Theorem \ref{thm:arithmetic_rational_cover}.

Suppose now that $q$ is not of the form $r^{1/m}$, with $r > 1$ a rational number and $m \in \N$.
From Theorem \ref{thm:intersection} it follows that $|\A_i \cap \G| \leq 6$ for all $i \in J$.
Then, for all $n \in \N$,
\begin{equation*}
n = |\G^{(n)}| \leq \sum_{i=1}^a |\A_i \cap \G^{(n)}| \leq \sum_{i=1}^a |\A_i \cap \G| \leq 6a ,
\end{equation*}
and so $a \geq n / 6$. This completes the proof of Theorem \ref{thm:arithmetic_irrational_cover}.

\subsection*{Acknowledgements}
The author is grateful to Paolo Leonetti (Universit\`a Bocconi), for having attracted his attention on a particular case of Theorem \ref{thm:geometric_cover}, and also to Salvatore Tringali (LJLL, Universit\'e Pierre et Marie Curie), for suggestions which improves the readability of the paper.

\bibliographystyle{amsalpha}
\providecommand{\bysame}{\leavevmode\hbox to3em{\hrulefill}\thinspace}
\providecommand{\MR}{\relax\ifhmode\unskip\space\fi MR }

\providecommand{\href}[2]{#2}

\end{document}